\documentclass{amsart}
\usepackage{amssymb}

\def\squareforqed{\hbox{\rlap{$\sqcap$}$\sqcup$}}
\def\qed{\ifmmode\squareforqed\else{\unskip\nobreak\hfil
\penalty50\hskip1em\null\nobreak\hfil\squareforqed
\parfillskip=0pt\finalhyphendemerits=0\endgraf}\fi\medskip}

\newcommand{\udot}{{}^{\textstyle .}}

\newcommand{\PSL}{\mathrm{PSL}}
\newcommand{\PGL}{\mathrm{PGL}}
\newcommand{\Mat}{\mathrm{M}}

\newcommand{\Sz}{\mathrm{Sz}}
\newcommand{\Co}{\mathrm{Co}}

\newcommand{\PSU}{\mathrm{PSU}}
\newcommand{\Or}{\mathrm{P\Omega}}

\newcommand{\He}{\mathrm{He}}
\newcommand{\MM}{\mathbb{M}}

\newtheorem{theorem}{Theorem}

\newtheorem{example}{Example}
\newtheorem{lemma}{Lemma}

\title{There is no $\Sz(8)$ in the Monster}
\date{Version of 18th August, 2015}
\author{Robert A. Wilson}

\address{School of Mathematical Sciences\\
Queen Mary University of London\\
London E1 4NS\\U.K.}
\email{R.A.Wilson@qmul.ac.uk}
\begin{document}
\maketitle

\begin{abstract}
As a  contribution to an eventual solution 
of the problem of the determination of the maximal subgroups of the Monster
we show that there is no subgroup isomorphic to $\Sz(8)$. The proof is largely,
though not entirely, computer-free.
\end{abstract}

\section{Introduction}
The Fischer--Griess Monster group $\mathbb M$ is the largest of the
$26$ sporadic simple groups, and was first constructed by Griess
\cite{Griess} in 1982. A simplified construction along the same general lines
was given by Conway \cite{Conway}.

One of the major problems in group theory today is that
of classifying the maximal subgroups of the finite simple groups and their automorphism
groups.
Much work has been done over many years attempting to determine the
maximal subgroups of  $\MM$, but it is still the only
sporadic group whose maximal subgroups are not completely
classified (see \cite{FSG} and references therein). 

The maximal $p$-local
subgroups of the Monster were classified in \cite{oddlocals,MSh,Meier}, and much theoretical work on
non-local subgroups was accomplished in \cite {Anatomy1,Anatomy2}.
Following successful computer constructions of the Monster \cite{3loccon,2loccon}
other techniques became available, and further progress was made
\cite{post,A5subs,S4subs,L241,L227,L213B}, including discovery of five previously unknown
maximal subgroups, isomorphic to $\PSL_2(71)$, $\PSL_2(59)$,
$\PSL_2(41)$, $\PGL_2(29)$, $\PGL_2(19)$.
 
The cases left open by this previous work are
possible maximal subgroups with socle isomorphic to one of the following simple groups:
$$\PSL_2(8), \PSL_2(13), \PSL_2(16), \PSU_3(4), \PSU_3(8), \Sz(8).$$
Of these, $\PSL_2(8)$ and $\PSL_2(16)$ have been classified in unpublished work
of P. E. Holmes. The case of $\Sz(8)$ is
particularly interesting because it is not yet known whether $\Sz(8)$ is a subgroup
of the Monster at all.

Throughout this paper, $\mathbb M$ denotes the Monster, and $S$ denotes a subgroup
of $\mathbb M$, isomorphic to $\Sz(8)$. The notation of the {\sc Atlas} 
\cite{Atlas} is generally used for group
names and structures, occasionally replaced by more traditional names as in \cite{FSG}.
 In addition, $B\cong 2^{3+3}{:}7$
denotes the Borel subgroup of $S$. 

The main result of this paper 
is the following.
\begin{theorem}\label{noSz8}
 There is no subgroup isomorphic to $\Sz(8)$ in the Monster sporadic simple group $\mathbb M$.
\end{theorem}

The structure of the proof is as follows. First we prove the following.
\begin{theorem}
\label{whereisB}
If $B\cong 2^{3+3}{:}7$ is a subgroup of $\mathbb M$ isomorphic to the Borel subgroup
of $\Sz(8)$, then $B$ lies in one of the maximal subgroups
$M_1$ of shape $2^{1+24}\udot\Co_1$ or $M_2$ of shape 
$2^3.2^6.2^{12}.2^{18}.(\PSL_3(2)\times 3S_6)$.
\end{theorem}
Then we look more closely at the $M_1$ case and prove the following.
\begin{theorem}
\label{noBinCo1}
The Conway group $\Co_1$ does not contain a subgroup isomorphic to the Borel
subgroup of $\Sz(8)$.
\end{theorem}
This reduces the $M_1$ case to a classification of certain subgroups of $2^{1+24}$, which
yields exactly three classes of $2^3$ which might lie in $B$.
%,
%and these three cases are distinguished by the class of $2^2$ which they contain.
%The second main result is about fusion of $2^2$ subgroups.
We then show that two of the three cases do not 
in fact extend to a copy of $B$.
%even extend to $2^{3+3}{:}7$. 
The other case may extend to $B$, but not to $S$.
First we show that the $M_2$ case reduces to this last $M_1$ case.
Then this possibility is eliminated %relatively easily, 
using a small computation of orbits in the Held group
to show that any group generated by the subgroup $2^3{:}7$ and an involution
inverting the $7$-element has non-trivial centralizer.

\section{Locating the Borel subgroup}
\label{where}
In this section, we consider all possibilities for known maximal subgroups of $\mathbb M$
which could contain a subgroup $B\cong 2^{3+3}{:}7$ extending to $\Sz(8)$ in $\mathbb M$.
It is shown in \cite{MSh,Meier} that every $2$-local subgroup of the Monster 
is contained in one of the known maximal subgroups. These papers do not however contain
the stronger assertion that every $2$-local subgroup of the Monster is contained in one of the
known $2$-local maximal subgroups. 
(That is, they classify $2$-local maximal subgroups, not maximal $2$-local subgroups.)
We therefore need first of all to consider the other
known maximal subgroups. A list of $43$ of the currently known $44$
classes of maximal subgroups can be found in Table 5.6
of \cite{FSG}: the subgroup $\PSL_2(41)$ found in \cite{L241}
was at that time thought not to exist.

\begin{lemma}\label{2locs} 
Every subgroup $B\cong 2^{3+3}{:}7$ of $\mathbb M$ lies in one of
the known $2$-local maximal subgroups.
\end{lemma}
\begin{proof}
It is easy to see that $B$ cannot lie in any of the known non-local subgroups. Most of the
$p$-local subgroups for $p$ odd are easy to eliminate, and we quickly reduce to those whose
non-abelian composition factors are $\mathrm{HN}$, $\mathrm{Fi}_{24}'$, 
$\mathrm{Th}$, $\Or_8^+(3)$
or $\Or_8^-(3)$. In the case of $\mathrm{HN}$, the subgroup lies in $2^3.2^2.2^6.\PSL_3(2)$,
which embeds into $2^3.2^6.2^{12}.2^{18}.\PSL_3(2)$ in the Monster. In the case of $\mathrm{Th}$, 
it either centralizes an involution,
or lies in $2^5\udot\PSL_5(2)$, in which case it again either centralizes an involution or
lies in the subgroup  $2^3.2^6.2^{12}.2^{18}.\PSL_3(2)$.
%normalizes the $2B^3$ of type (a).

The group $\Or_8^-(3)$ reduces to $\Or_7(3)$, which does not contain a subgroup
isomorphic to $B$. Similarly $\Or_8^+(3)$ reduces to either $\Or_7(3)$ or $\Or_8^+(2)$,
and the latter reduces to $2^6A_8$. Moreover, since a triality automorphism of $\Or_8^+(2)$
is realised in the Monster, this determines the group $2^6A_8$ up to conjugacy, and it is
easily seen to lie in $2^{10+16}\udot\Or_{10}^+(2)$.

We are left with $3\udot\mathrm{Fi}_{24}'$. 
By inspection, all $2$-local maximal subgroups thereof
lie inside $2$-local maximal subgroups of $\mathbb M$, which leaves $\mathrm{Fi}_{23}$
and $\Or_{10}^-(2)$ to consider. Similar arguments in these groups rapidly conclude
the proof.
\end{proof}
The next lemma is a restatement of Theorem~\ref{whereisB}.
\begin{lemma}\label{2locs} 
Every subgroup $B\cong 2^{3+3}{:}7$ of $\mathbb M$ lies in one of
the two maximal subgroups $M_1=2^{1+24}\Co_1$ or 
$M_2=2^3.2^6.2^{12}.2^{18}(\PSL_3(2)\times 3S_6)$.
\end{lemma}
\begin{proof}
Since $B$ is generated by elements of order $7$, we reduce in each case to the
normal subgroup of the relevant maximal subgroup, generated by the elements of order $7$.
This allows us to eliminate the case $2^2.2^{11}.2^{22}.\Mat_{24}$, which is contained in
$M_1$; and the case $2^2.{}^2E_6(2)$, which is contained in $2.\mathbb B$; 
and the case $2^5.2^{10}.2^{20}.\PSL_5(2)$, which is contained in
$2^{10}.2^{16}.\Or_{10}^+(2)$.

Now inside $2\udot\mathbb B$ we easily reduce to $2$-local maximal subgroups,
most of which are contained in one of the $2$-constrained maximal
subgroups of the Monster considered above. Everything else reduces to $2^2.{}^2E_6(2)$.
But %now $B$ would be embedded in 
modulo the centre this is a group of Lie type in characteristic $2$, whose maximal $2$-local
subgroups are given by the Borel--Tits theorem, 
and again lie in the $2$-constrained subgroups above.

Finally, we eliminate $2^{10}.2^{16}.\Or_{10}^+(2)$. The centralizer of the
action of $2^{3+3}$ on the $2^{10}$
orthogonal space must be a singular subspace, whose radical is acted on by the
element of order $7$. Now singular vectors are in class $2B$, while non-singular vectors are
in class $2A$. As a module for $7$, therefore, this radical is either irreducible
$3$-dimensional, in which case $B$ is contained in $M_2$, or contains fixed points,
in which case $B$ is contained in $M_1$.
\end{proof}

\section{Subgroups of $\Co_1$ isomorphic to $2^{3+3}{:}7$}
\label{Co1sub}
We begin with the $M_1$ case. As a first step,
in this section we prove Theorem~\ref{noBinCo1}, that $\Co_1$ does not contain a
subgroup isomorphic to $B$. We use the list of maximal subgroups given in \cite{FSG}, and more particularly the $2$-local maximal subgroups classified by Curtis \cite{Curtis}. Further information
about maximal subgroups is taken from the {\sc Atlas} \cite{Atlas}.
\begin{lemma}\label{C2a}
Any subgroup of $\Co_1$ isomorphic to the Borel subgroup of $\Sz(8)$ lies in
a conjugate of $2^{1+8}\udot\Or_8^+(2)$. Moreover, the subgroup $2^3{:}7$
is determined up to conjugacy.
\end{lemma}
\begin{proof}
Inspection of the list of maximal subgroups of $\Co_1$, as well as maximal subgroups of
maximal subgroups, and so on as far as necessary, shows that any $2^{3+3}{:}7$ in
$\Co_1$ lies in one of the maximal $2$-local subgroups $2^{1+8}\udot\Or_8^+(2)$ or
$2^{2+12}(A_8\times S_3)$ or $2^{11}{:}\Mat_{24}$. But it is easy to see that in the latter two cases any such subgroup centralizes an involution of $\Co_1$-class $2A$, so reduces to 
the first case.

Now $\Or_8^+(2)$ does not contain $2^{3+3}{:}7$, so we must have a $2^3$ subgroup
of $2^{1+8}$. This corresponds to a totally isotropic $3$-space in the orthogonal $8$-space.
All such $3$-spaces are equivalent. Each such $3$-space has stabilizer $2^{3+6}{:}\PSL_3(2)$ in
$\Or_8^+(2)$, so up to conjugacy there is a unique $7$ normalizing it.
\end{proof}

Indeed, the full pre-image of this $3$-space stabilizer in $2^{1+8}\udot\Or_8^+(2)$ lies
inside the octad stabilizer in $2^{11}{:}\Mat_{24}$. 
Since the latter is a split extension, it is much easier to calculate in than the involution
centralizer itself. The $2^3$ itself consists of octads which are disjoint from the
fixed octad.

\begin{lemma}\label{notM24}
The subgroup $2^{11}{:}\Mat_{24}$ of $\Co_1$ does not contain a group isomorphic to $B$. 
\end{lemma}
\begin{proof}
The relevant subgroup of $2^{11}{:}\Mat_{24}$ is
$2^{11}{:}2^{1+6}\PSL_3(2)$, that is the preimage of the involution centralizer
in $\Mat_{24}$. This is contained in $2^{11}{:}2^4A_8$, in which the $2^{11}$ is a uniserial
modiule for $2^4A_8$, with factors $1+4+6$. As a module for the cyclic group of order $7$,
therefore, the $2^{11}$ has structure
$1a+1a+3a+3a+3b$, and the $2^3$ which corresponds to the isotropic $3$-space is
one of the copies of $3a$. It follows that in a putative $2^{3+3}{:}7$, the top $2^3$ is
also of type $3a$. This identifies the $2^3{:}7$ up to conjugacy in $2^{1+6}{:}\PSL_3(2)$.

The module structure of the $2^{11}$ for this group $2^3{:}7$ can now be calculated.
There can be gluing of a $3b$ under a $3a$, or of a $1a$ under a $3b$, but gluing a $3a$
under anything else is impossible. It follows that we can quotient by $1a+3b$, to get a
group $2^6{:}2^3{:}7$ in which all three $2^3$ chief factors are of type $3a$. 
A straightforward calculation now reveals that this group does not contain a copy of
the group $2^{3+3}{:}7$ we are seeking.
\end{proof}

This concludes the proof of Theorem~\ref{noBinCo1}.

\section{Pure $2^3$ subgroups of $2^{1+24}$}
\label{C2B}
By this stage we know that any embedding of $B$ in $M_1$ involves a $2^3$ in $2^{1+24}$,
and a quotient $2^3{:}7$ in $\Co_1$.
We next show that this forces the $7$-elements to be in $\mathbb M$-class $7A$
(corresponding to $\Co_1$-class $7B$).
\begin{lemma}\label{7A}
There is no $2^3{:}7$ in $\Co_1$ containing elements of $\Co_1$-class $7A$.
\end{lemma}
\begin{proof} The $7A$-elements in $\Co_1$ are fixed-point-free
in the action of $2\udot \Co_1$ on the Leech lattice. If they lie in $2^3{:}7$ in $\Co_1$,
then this lifts to $2^3{:}7$ in $2\udot\Co_1$, acting faithfully on the Leech lattice.
But then the element of order $7$ would have a fixed point, which is a contradiction.
This concludes the proof. 
\end{proof}

\begin{lemma}\label{lem:inCo1}
 There are exactly
three conjugacy classes of $2^3{:}7$  in $2^{1+24}\Co_1$ that have 
the properties that the $2^3$ lies in $2^{1+24}$ and the $7$-element
lies in $\Co_1$-class $7B$. Their
centralizers in $2^{1+24}\Co_1$ are respectively 
\begin{enumerate}
\item $2^{1+6}S_4$, 
\item $2^{1+6}.7$, and 
\item $2^{1+6}.2^2$.
\end{enumerate}
\end{lemma}
\begin{proof}
First, the $7B$-normalizer in $\Co_1$ is $(7{:}3\times \PSL_3(2)){:}2$, in which 
the two factors $7{:}3$ and $\PSL_3(2)$ both have two $3$-dimensional representations,
which we will denote $3a$ and $3b$. Then the representation of
$7{:}3\times \PSL_3(2)$ on the $2^{24}$ is
$$1\otimes 3a+1\otimes 3b+ 3a\otimes 3a+3b\otimes 3b.$$
Since the outer half of the $7$-normalizer swaps $3a$ with $3b$, we may assume
that our $2^3$ lies in the $3a\otimes 3a$ part of the representation.

Now we may interpret our $7$-element as a scalar in the field $\mathbb F_8$
of order $8$, so that $3a\otimes 3a$ becomes a $3$-space over $\mathbb F_8$.
Then we  classify the orbits of $\PSL_3(2)$ on the $(8^3-1)/(8-1)=73$
one-dimensional subspaces of this $3$-space. This is a straightforward calculation, and we
find that the orbit lengths are $7$, $24$, and $42$. Thus there are exactly
three conjugacy classes of $2^3{:}7$ of this kind in $2^{1+24}\Co_1$, with
centralizers respectively $2^{1+6}S_4$, $2^{1+6}.7$, and $2^{1+6}.2^2$.
\end{proof}

\section{Examples}
\label{examples}
The $2B$-elements in $2^{1+24}$, modulo the central involution,
correspond to crosses in the Leech lattice, that is congruence classes modulo $2$
of lattice vectors of type $4$. The $2^3$ subgroups described in Lemma~\ref{lem:inCo1}
can therefore be described by representative vectors of three such classes.
We use the octonionic notation of \cite{octoLeech} for the Leech lattice, and explicit
generators for the Conway group given in \cite{octoConway}. In particular, we take
the $7$-element to rotate the imaginary units as $i_t\mapsto i_{t+1}$, with subscripts
read modulo $7$, and the $\PSL_3(2)$ to be generated modulo the central involution
of $2\udot\Co_1$ by the sign-changes and permutations on the three octonionic coordinates,
together with the matrix 
$$g_1=\begin{pmatrix}0&\overline{s}&\overline{s}\cr s&-1&1\cr s&1&-1\end{pmatrix}$$
acting by right-multiplication on row vectors.

Now if $\PSL_3(2)$ acts in the usual way on $\mathbb F_2^2$, and $\eta$ is a root of
$x^3+x+1$ modulo $2$, then the three orbits on $1$-spaces have representatives
respectively $(1,0,0)$, $(1,\eta,0)$ and $(1,\eta,\eta^2)$, giving orbit lengths
$7$, $42$ and $24$ respectively. This can be translated directly into the above
situation, and enables us to write down representatives for the three orbits of
$2^3{:}7$ described in Lemma~\ref{lem:inCo1}.

%\paragraph{Example 1}
\begin{example}
In the first case, the $2^3$ is centralized by an $S_4$ in the $\PSL_3(2)$, and 
this $S_4$ belongs to the so-called Suzuki chain of subgroups, and centralizes $A_8$.
The resulting subgroup $S_4\times A_8$ lies in the stabilizer of a trio of three disjoint
octads. 
We may take the $7B$-element to cycle the imaginary units $i_0,i_1,\ldots, i_6$
in the obvious way, and the $2^3$ to consist of the crosses defined by the vector
$2(-1+i_0+i_1+i_3,0,0)$ and its images under the $7$-cycle. 

Adjoining the central involution of $2^{1+24}$ and the cross defined by $(4,0,0)$
gives a copy of the $2^5$ with normalizer $2^5.2^{10}.2^{20}.(\PSL_5(2)\times S_3)$. 
In particular, any copy of $B$ containing this $2^3{:}7$ also lies in $M_2$.

By applying the matrix $g_1$ we obtain a spanning set for the $3$-space over $\mathbb F_8$.
A second basis vector may be taken modulo $2$ to be
$(-2-i_0+i_3+i_5+i_6)(1,1,0)$.
% Let us order the $24$ coordinates in the usual way in three blocks of eight, corresponding to
%disjoint octads of the Steiner system $S(5,8,24)$, in such a way that our element
%in class $7B$ fixes the first point in each block and cycles the remaining $7$.
%Then it is possible to generate a $2^3$ of this type with congruence classes represented
%by the $7$ cyclic images of $(0,4,4,4,0,4,0,0)$ in any block.
\end{example}
%\paragraph{Example 2}
\begin{example}
In the second case, the $2^3$ is centralized by an element of order $7$. This is necessarily of
$\Co_1$ class $7B$, so can be conjugated to the element of class $7B$ described in the 
previous example. This element centralizes a $2^{1+6}$ in $2^{1+24}$, which is acted on
by a group $\PSL_3(2)$ which identifies the two invariant $2^3$ subgroups. We can take either
of them, since they are interchanged by an automorphism which inverts the $7B$-element.

With the same notation as above, we find that an example is generated by
the congruence classes of %$(8,0^{23})$ and $(2^8;4,0^7;4,0^7)$, 
$(4,0,0)$ and $2(\overline{s},1,\pm1)$,
and images under 
permutations of
the three octads.
\end{example}
%\paragraph{Example 3}
\begin{example}\label{ex3}
We make the third example directly by translating $(1,\eta,0)$ into octonionic language,
so that it is again normalized by the canonical element of order $7$.
%Another $2^3$ normalized by the same $7B$-element 
It can be generated by the
images of the congruence class of 
the vector 
$(-2-i_0+i_3+i_5+i_6,2i_4+i_0+i_3-i_5+i_6,0)$.
%$(1+i_0)(s-2,s,s)$.
%$$(-2,4,2,2,0,2,0,0;2,0,2,2,0,2,0,0;2,0,2,2,0,2,0,0).$$
%It is not immediately obvious that this is in a different conjugacy class from the
%previous two examples. We shall prove this shortly, by showing that the $2^2$ subgroups
%are different in the three cases.
\end{example}

\section{Identifying the $2^2$ subgroups}
\label{4groups}
It is well-known \cite{MSh}
that there are three
classes of $2^2$ of pure $2B$-type in the Monster, with the following properties
with respect to the centralizer $2^{1+24}\Co_1$ of any one of its involutions.
\begin{itemize}
\item[(a)] Contained in the normal subgroup $2^{1+24}$, so having
centralizer of the shape $(2\times 2^{1+22}).2^{11}\Mat_{24}$.
\item[(b)] Mapping onto an element of $\Co_1$-class $2A$, whose centralizer is $\Co_1$
has shape
$2^{1+8}\udot \Or_8^+(2)$. The centralizer of this $2^2$-group in the
Monster is however only $(2^9\times 2^{1+6}).2^{1+8}.2^6A_8$.
\item[(c)] Mapping onto an element of $\Co_1$-class $2C$, which has centralizer
$2^{11}\Mat_{12}.2$ in $\Co_1$. The centralizer of this $2^2$ in the Monster is
$2^{12}.2^{11}.\Mat_{12}.2$.
\end{itemize}

It is also proved in \cite{MSh}, and is in any case a straightforward calculation,
that all three of these $2B^2$ subgroups are represented in $2^{1+24}$ modulo
its centre, and that there is a unique conjugacy class in each case.
In standard notation, if one of the involutions is taken to be the congruence class of $(8,0^{23})$, then the
other is the congruence class of either $(4^4,0^{20})$ or $(2^8,4^2,0^{14})$ or
$(2^{12},4,0^{11})$. These are of type (a), (b), (c) respectively.
Examples in octonionic notation are $(4,0,0)$ with respectively $2(1+i_0+i_1+i_3,0,0)$ or
$2(\overline{s},1,1)$ or $(1+i_0)(s-2,s,s)$.
From this it is immediate that in the first two cases in Lemma~\ref{lem:inCo1}
the $2^2$-subgroups are respectively
of type (a) and (b). 
A small calculation establishes that in the third case they are of type (c).
As this calculation is somewhat tricky to carry out accurately, we give a sketch here.

\begin{lemma}\label{typec}
The $2^3$ of type (3)  in Lemma~\ref{lem:inCo1} contains $2^2$ subgroups of type (c).
\end{lemma}
\begin{proof}
Let us take the example given in Lemma~\ref{ex3}
above, spanned by the congruence classes of the vectors
\begin{eqnarray*}
&&(-2-i_0+i_3+i_5+i_6,2i_4+i_0+i_3-i_5+i_6,0)\cr
&&(-2-i_1+i_4+i_6+i_0,2i_5+i_1+i_4-i_6+i_0,0).
\end{eqnarray*}
We aim to apply elements of the Conway group which map the first vector to 
a vector in the congruence class of $(4,0,0)$. First multiply the second and third coordinates
by $i_4$, then $i_6$, then $i_5$, then $i_1$ to get
\begin{eqnarray*}
&&(-2-i_0+i_3+i_5+i_6,-2-i_0+i_3+i_5+i_6,0)\cr
&&(-2-i_1+i_4+i_6+i_0,-2i_0-1+i_2-i_3+i_5,0).
\end{eqnarray*}
Now we can apply the matrix 
$$\frac12\begin{pmatrix}-1&1&s\cr 1&-1&s\cr \overline{s}&\overline{s}&0\end{pmatrix}$$
to obtain
\begin{eqnarray*}
&&-2(0,0,i_0+i_3+i_5+i_6)\cr
&&\frac12(1-3i_0+i_1+i_2-i_3-i_4+i_5-i_6,\cr
&&\qquad-1+3i_0-i_1-i_2+i_3+i_4-i_5+i_6,\cr
&&\qquad1-i_0-i_1-i_2-i_3-i_4-i_5-5i_6)
\end{eqnarray*}
We may, although this is not strictly necessary, tidy this up a little
by multiplying the second and third coordinates
by $i_0$ and then $i_1$, to obtain
\begin{eqnarray*}
&&2(0,0,1+i_1-i_2+i_4)\cr
&&\frac12(1-3i_0+i_1+i_2-i_3-i_4+i_5-i_6,\cr
&&\qquad-1-i_0-3i_1-i_2-i_3-i_4+i_5+i_6,\cr
&&\qquad1-i_0+i_1-i_2+i_3+5i_4+i_5-i_6)
\end{eqnarray*}
and finally multiply by $(1-i_1)$ and then $(1+i_2)/2$ to obtain
\begin{eqnarray*}
&&(0,0,4)\cr
&&(-i_0+i_2+i_3-i_6,-1-i_0-i_2-i_4,2-i_1-i_2+i_3+i_4)
\end{eqnarray*}
It is readily checked that this last vector lies in the Leech lattice, and that these two
congruence classes determine a $2B^2$ subgroup of type (c) in the Monster.
\end{proof}
\section{Eliminating the second and third cases}
In these two cases we show that there is no embedding of $B$ in $M_1$.
\begin{lemma}
The group $2^3{:}7$ of type (2) considered in Lemma~\ref{lem:inCo1}
cannot occur in a %$\Sz(8)$ 
copy of $B$ in the Monster.
\end{lemma}
\begin{proof}
The second type of $2^3$
has normaliser with order divisible by $7^2$, and lying in
$2^{1+24}\udot\Co_1$. Now the only maximal subgroups of $\Co_1$ whose order is
divisible by $7^2$ are $7^2{:}(3\times 2A_4)$ and $(A_7\times \PSL_3(2)){:}2$.
Since neither of these groups contains $2^3{:}7$, the group in question cannot
extend to $2^{3+3}{:}7$. %, so cannot extend to $\Sz(8)$.
\end{proof}

\begin{lemma}\label{lem:not22c1}
The $2^3{:}7$ subgroup of type (3) in Lemma~\ref{lem:inCo1}
cannot occur in a copy of $B$ in the Monster.
\end{lemma}
\begin{proof}
The third type of $2B$-pure $2^2$ 
has centralizer $(2^2 \times 2^{1+20})\udot \Mat_{12}\udot 2$, which lies
entirely within $2^{1+24}\udot\Co_1$.
Again, the subgroup $2^{3+3}{:}7$ of our putative $\Sz(8)$
projects onto %(at least) 
a subgroup $2^3{:}7$ of $\Co_1$. Moreover, the normal $2^{3+3}$ is 
in the centralizer of our $2^2$ and projects to a
pure $2^3$ subgroup of $\Co_1$. Since this $2^3$
lies in $\Mat_{12}.2$, we need to look at the embedding of $\Mat_{12}{:}2$ in $\Co_1$.
We have that the classes $2A$ and $2C$ in $\Mat_{12}{:}2$ fuse to $\Co_1$-class $2B$,
while $\Mat_{12}$-class $2B$ fuses to $\Co_1$-class $2A$. But there is no pure $2^3$ of
$\Co_1$-class $2B$, and no pure $2^3$ of $\Mat_{12}$-class $2B$. Therefore this case
cannot arise.
\end{proof}

%\section{The $2^2$ of $\Mat_{24}$-type}
\section{Eliminating the first case}
\label{M24type}
In this case we adopt a different strategy, and show that any subgroup of $\mathbb M$
which is generated by a $2^3{:}7$ of this type and an involution which inverts
an element of order $7$ therein has non-trivial centralizer. Since $\Sz(8)$ can be
generated in this way, and it is already known that every $\Sz(8)$ in $\mathbb M$ has
trivial centralizer, this proves that this $2^3{:}7$ cannot lie in $\Sz(8)$.
%The first type of $2B$-pure four-group in $\mathbb M$
%has centralizer $2^2.2^{11}.2^{22}.\Mat_{24}$ and can be easily
%disposed of. 

Before we prove this, we show that the $M_2$ case also reduces to this case.
\begin{lemma}\label{lem:notinN23}
%If $S\cong\Sz(8)$ is a subgroup of $\MM$, then
Every copy of the group $B\cong2^{3+3}{:}7$  in $M_2$
%$2^3.2^6.2^{12}.2^{18}.(\PSL_3(2)\times 3S_6)$ 
contains the socle of $M_2$.
%which is a $2^3$ of type (a) in Lemma~\ref{lem:inCo1}.
\end{lemma}
\begin{proof}
If we label the two $3$-dimensional representations of $\PSL_3(2)$ as
$3a$ and $3b$, and label other representations by their degrees, then
the representations of $\PSL_3(2)\times 3S_6$ on the chief factors of $N(2^3)$ are
respectively $3a\otimes 1$, $1\otimes 6$, $3b\otimes 4$, and $3a\otimes 6$.
Now $3a$ and $3b$ remain distinct on restriction to the subgroup of order $7$.
But in $B\cong 2^{3+3}{:}7$, the $3$-dimensional representations of the group of order
$7$ on $B''$ and $B'/B''$ are the same, and this can only 
occur in $M_2$ %$2^3.2^6.2^{12}.2^{18}.(\PSL_3(2)\times 3S_6)$ 
in the case when
$B$ contains the socle.
%The result now follows from Lemma~\ref{lem:not22a}.
\end{proof}

\begin{lemma}\label{lem:not22a}
The $2^3{:}7$ subgroup of type (1) in Lemma~\ref{lem:inCo1}
cannot occur in a copy of $\Sz(8)$ in the Monster.
%If $S\cong\Sz(8)$ is a subgroup of $\MM$, then
%the $2^2$-subgroups of $S$ are not of type (a).
\end{lemma}
\begin{proof}
%We have already seen that in this case there is a single possibility for the $2^3$, that is
%the one whose normalizer is $M_2$. It follows that there is a unique
%possibility for 
In this case the $2^3{:}7$  has centralizer $2^6{:}3S_6$, visible in $M_2$.
%Since the $7$-element is in class $7A$, it 
The $7$-element extends to exactly 266560
groups $D_{14}$ inside the invertilizer $(7\times \He){:}2$. It is easy to
calculate (using a suitable computer algebra package such as %Magma, 
GAP \cite{GAP}) %, or the Meataxe)
the orbits of $2^6{:}3S_6$ on these 266560 points, and to observe that there is
no regular orbit. (This permutation representation was taken from \cite{webatlas}.)
Now $\Sz(8)$ can be generated by subgroups $2^3{:}7$ and $D_{14}$ intersecting
in $7$.
It follows that if $\Sz(8)$ is generated by one of these
amalgams, with this particular $2^3{:}7$,
then it is centralized by a non-trivial element.
This is a contradiction. %, which proves the lemma.
\end{proof}

This concludes the proof of Theorem~\ref{noSz8}.

%\affiliationone{Robert A. Wilson\\
%School of Mathematical Sciences\\
%Queen Mary, University of London\\
%Mile End Road\\
%London E1 4NS\\U.K.
%\email{R.A.Wilson@qmul.ac.uk}}


\begin{thebibliography}{99}
%\bibitem{Bray} J. N. Bray, { An improved method for generating the centralizer of an
%involution}, {\it Arch. Math. (Basel)} {\bf 74} (2000), 241--245.

\bibitem{Conway} J. H. Conway, A simple construction for the Fischer--Griess monster
group, {\it Invent. Math.} {\bf 79} (1985), 513--540.

\bibitem{Atlas}
J. H. Conway, R. T. Curtis, S. P. Norton, R. A. Parker and R. A. Wilson,
{\it An Atlas of Finite Groups}, Oxford University Press, 1985.

\bibitem{Curtis} R. T. Curtis, On subgroups of $\cdot0$. II. Local structure,
{\it J. Algebra} {\bf 63} (1980), 413--434.

\bibitem{GAP} GAP---Groups, Algorithms, Programming, 

\bibitem{Griess} R. L. Griess, The Friendly Giant, {\it Invent. Math.} {\bf 69} (1982), 1--102.

\bibitem{S4subs} P. E. Holmes, {A classification of subgroups of the Monster
isomorphic to  $S_4$ and an application,} {\it J. Algebra} {\bf 319} (2008), 3089--3099.

\bibitem{2loccon} P. E. Holmes and R. A. Wilson, {A new computer construction
of the Monster using $2$-local subgroups}, {\it J. London Math. Soc.} {\bf 67} (2003), 349--364.

\bibitem{post} P.E. Holmes and R. A. Wilson,
{ A new maximal subgroup of the Monster},
{\it J. Algebra} {\bf 251} (2002), 435--447.

\bibitem{A5subs} P. E.  Holmes and R. A. Wilson, {On subgroups of the Monster
containing $A_5$'s,} {\it J. Algebra} {\bf 319} (2008), 2653--2667.

\bibitem{3loccon} S. A. Linton, R. A. Parker, P. G. Walsh and R. A. Wilson,
A computer construction of the Monster, {\it J. Group Theory} {\bf 1}
(1998), 307--337.

\bibitem{MSh} U. Meierfrankenfeld and S. V. Shpectorov,
Maximal 2-local subgroups of the Monster and Baby Monster, Preprint, Michigan State University, 2002.
http://www.mth.msu.edu/\~{}meier/Preprints/2monster/maxmon.pdf

\bibitem{Meier} U. Meierfrankenfeld,
The maximal 2-local subgroups of the Monster and Baby Monster, II,
Preprint, Michigan State University, 2003.
http://www.mth.msu.edu/\~{}meier/Preprints/2monster/2MNC.pdf

\bibitem{Anatomy1} S. Norton, { Anatomy of the Monster: I,} in
{\it Proceedings of the
{\sc Atlas} Ten Years On conference (Birmingham 1995)}, pp.\ 198-214, Cambridge
Univ.\ Press, 1998.

\bibitem{Anatomy2} S. P. Norton and R. A. Wilson, Anatomy of the Monster: II,
{\it Proc. London Math. Soc.} {\bf 84} (2002), 581--598.

\bibitem{L241} S. P. Norton and R. A. Wilson,
A correction to the $41$-structure of the Monster, 
a construction of a new maximal subgroup $L_2(41)$, and a new Moonshine
phenomenon,
{\it J. London Math. Soc.} {\bf 87} (2013), 943--962.

\bibitem{oddlocals} R. A. Wilson, The local subgroups of the Monster,
{\it J. Austral. Math. Soc. (A)} {\bf 44} (1988), 1--16.

\bibitem{FSG} R. A. Wilson,
{\it The finite simple groups}, Springer GTM 251, 2009.

\bibitem{octoLeech} R. A. Wilson, Octonions and the Leech lattice,
{\it J. Algebra} {\bf 322} (2009), 2186--2190.

\bibitem{octoConway} R. A. Wilson,
Conway's group and octonions, {\it J. Group Theory} {\bf 14} (2011), 1--8.

\bibitem{L227} R. A. Wilson,
Classification of subgroups isomorphic to $\PSL_2(27)$ in the Monster,
{\it LMS J. Comput. Math.} {\bf 17} (2014), 33--46.

\bibitem{L213B} R. A. Wilson,
Every $\PSL_2(13)$ in the Monster contains $13A$-elements,
{\it LMS J. Comput. Math.}, to appear.
%Preprint, QMUL, 2015.

\bibitem{webatlas} R. A. Wilson et al., {\it An Atlas of Group Representations},
http://brauer.maths.qmul.ac.uk/Atlas/.

\end{thebibliography}
\end{document}